\documentclass[a4paper]{article}
\usepackage{amsthm,amsfonts,amsmath,amssymb,units}
\usepackage[abbrev,nobysame]{amsrefs}
\usepackage[cp1251]{inputenc}
\usepackage[english]{babel}
\usepackage[final]{graphicx}
\usepackage{setspace}
\usepackage[12pt]{extsizes}
\begin{document}
\newcommand{\per}{{\rm per}}
\newtheorem{teorema}{Theorem}
\newtheorem{lemma}{Lemma}
\newtheorem{utv}{Proposition}
\newtheorem{svoistvo}{Property}
\newtheorem{sled}{Corollary}
\newtheorem{con}{Conjecture}
\newtheorem{zam}{Remark}
\newtheorem{quest}{Question}

\author{A. A. Taranenko}
\title{Vertices of the polytope of polystochastic matrices and product constructions\thanks{This work was funded by the Russian Science Foundation under grant No 22-21-00202, https:// rscf.ru/project/22-21-00202/.
}}
\date{}

\maketitle

\begin{abstract}
A multidimensional nonnegative matrix is called  polystochastic if the sum of its entries at each line is equal to $1$. The set of all polystochastic matrices of order $n$ and dimension $d$ is a convex polytope $\Omega_n^d$. In the present paper, we compare known bounds on the number $V(n,d)$ of vertices of the polytope $\Omega_n^d$, propose two constructions of vertices of $\Omega_n^d$ based on multidimensional matrix multiplication, and list all vertices of the polytope $\Omega_3^4$.
\end{abstract}

\section{Definitions}

A \textit{$d$-dimensional matrix $A$ of order $n$} is an array $(a_\alpha)_{\alpha \in I^d_n}$, $a_\alpha \in\mathbb R$, $I_n^d = \{  \alpha = (\alpha_1, \ldots, \alpha_d) | \alpha_i \in \{ 1, \ldots, n\}\}$.  A matrix $A$ is called \textit{nonnegative} if all $a_{\alpha} \geq 0$, and it is called a \textit{$(0,1)$-matrix} if all its entries are $0$ or $1$. The \textit{support} $supp(A)$ of a matrix $A$ is the set of all indices $\alpha$ for which $a_{\alpha} \neq 0$. Denote the  cardinality of the support $A$ by $N(A)$.

Let $k\in \left\{0,\ldots,d\right\}$. A \textit{$k$-dimensional plane} in $A$ is the submatrix of $A$ obtained by fixing $d-k$ indices and letting the other $k$ indices vary from 1 to $n$. A $1$-dimensional plane is said to be a \textit{line}, and a $(d-1)$-dimensional plane is a \textit{hyperplane}. Matrices $A$ and $B$ are called \textit{equivalent} if one can be obtained from the other by permutations of parallel hyperplanes and/or transposes.

A multidimensional nonnegative matrix $A$ is called \textit{polystochastic} if the sum of its entries at each line is equal to $1$. Polystochastic matrices of dimension $2$ are known as doubly stochastic matrices.  Doubly stochastic $(0,1)$-matrices are exactly the permutation matrices. $d$-Dimensional polystochastic $(0,1)$-matrices for $d \geq 3$ are said to be \textit{multidimensional permutations}.

It is easy to see that the set of $d$-dimensional polystochastic matrices of order $n$ is a convex polytope that we denote as $\Omega_n^d$.  We will say that a matrix $A$ is a  \textit{convex combination} of  matrices $B_1, \ldots, B_k$ if there exist $\lambda_1, \ldots, \lambda_k$, $0 < \lambda_i < 1$, $\lambda_1 + \cdots + \lambda_k = 1$ such that 
$$A = \lambda_1 B_1 + \cdots + \lambda_k B_k.$$ 

 A matrix $A \in \Omega_n^d$ is a \textit{vertex} of the Birkhoff polytope $\Omega_n^d$ if  for every convex combination $A = \lambda_1 B_1 + \cdots + \lambda_k B_k$ we have that all $B_i$ are equal to $A$. In other words, vertices  of  the  polytope $\Omega_n^d$ cannot be expressed as  a nontrivial convex combination of elements of $\Omega_n^d$.
Denote by $V(n,d)$ the number of vertices of the polytope $\Omega_n^d$.

\section{Preliminary results}

From the definition, it is not hard to deduce the following geometric properties of the polytope $\Omega_n^d$. 

\begin{utv} \label{polytopedim}
The polytope $\Omega_n^d$ is a $(n-1)^d$-dimensional polytope in $\mathbb{R}^{n^d}$ with $n^d$ facets. 
\end{utv}

For the first time, the problem of determining all vertices of the polytope of polystochastic matrices  was  raised  by Jurkat and Ryser~\cite{JurRys.stochmatr}.

It is easy to see that every multidimensional permutation is a vertex of $\Omega_n^d$. The well-known Birkhoff theorem states that every vertex of the polytope of doubly stochastic matrices is a permutation matrix.  

The polytope of $d$-dimensional polystochastic matrices of order $2$ can be completely described (see, for example,~\cite{my.obzor}).  From  the description it follows that  all vertices of $\Omega_2^d$ are multidimensional permutations.

The smallest example of a vertex different from a multidimensional permutation is the following $3$-dimensional polystochastic matrix of order $3$. It was discovered independently by many researchers~\cite{ChanPakZha.polystoch,csima.multmatr, CuiLiNg.Birkfortensor,LinLur.birvert, FichSwart.3dimstoch, KeLiXiao.extrpoints, WangZhang.permtensor}.

$$
V_3^3 = \left( \begin{array}{ccc|ccc|ccc}
1 & 0 & 0 & 0 & \nicefrac{1}{2} & \nicefrac{1}{2} & 0 & \nicefrac{1}{2} & \nicefrac{1}{2} \\
0 & \nicefrac{1}{2} & \nicefrac{1}{2} & \nicefrac{1}{2} & \nicefrac{1}{2} & 0 & \nicefrac{1}{2} & 0 & \nicefrac{1}{2} \\
0 & \nicefrac{1}{2} & \nicefrac{1}{2} & \nicefrac{1}{2} & 0 & \nicefrac{1}{2} & \nicefrac{1}{2} & \nicefrac{1}{2} & 0 \\
\end{array}
\right).
$$

The computer search shows that there are 12 multidimensional permutations and 54 matrices equivalent to the above~\cite{FichSwart.3dimstoch, ChanPakZha.polystoch}. All vertices of the polytope $\Omega^3_4$ were also found in~\cite{KeLiXiao.extrpoints}.

The following basic facts on vertices of the polytope of polystochastic matrices were established by Jurkat and Ryser~\cite{JurRys.stochmatr}.

\begin{utv}[\cite{JurRys.stochmatr}]
Let $A$  be  a matrix from  $\Omega_n^d$. If there is a $d$-dimensional matrix $B$ of order $n$ such that $supp(B) \subset supp(A)$ and every line of $B$ sums to $0$, then $A$ is  not a vertex.
 \end{utv}
 
 \begin{utv}[\cite{JurRys.stochmatr}]
 A polystochastic matrix $A$ is a vertex if and only if there is a unique polystochastic matrix with support equal to $supp(A)$.
 \end{utv}

\begin{utv}[\cite{JurRys.stochmatr}] \label{vertexcrit}
A polystochastic matrix $A$ is a vertex of $\Omega_n^d$ if and only if the incidence matrix $I(A)$ between the elements of $supp(A)$ and lines of $A$ has a full rank. 
\end{utv}

From these results, we have  the following bound on the cardinality $N(A)$ of the support of vertices $A$ of  $\Omega_n^d$.

\begin{utv} \label{supportbound}
Let $A$ be a vertex of the polytope of $d$-dimensional polystochastic matrices  of order $n$. Then the cardinality of the support of $A$
$$N(A) \leq n^d - (n-1)^d.$$
\end{utv}

For $3$-dimensional polystochastic matrices, it was  proved in~\cite{JurRys.stochmatr}, and in~\cite{my.obzor} it was generalized for nonnegative multidimensional matrices with equal sums in all $k$-dimensional planes.

The question of the minimum cardinality of the support of a vertex that is different from a permutation for $3$-dimensional polystochastic matrices was raised in~\cite{FichSwart.3dimstoch}.

\section{Bounds on the number of vertices}

McMullen in~\cite{McMullen.facesofpolytope} found the maximal  number of $i$-dimensional faces of polytope with a given number of vertices. As a consequence of this result, one can estimate the number of vertices of a polytope with a given number of facets.

\begin{utv}[see, e.g., \cite{Bron.covexpoly}]
The number of vertices $V$ of a convex $m$-dimensional polytope with $k$ facets, $k \geq m$, is
$$V \leq {k - \lfloor \frac{m+1}{2} \rfloor \choose k - m} +  {k - \lfloor \frac{m+2}{2} \rfloor \choose k - m}. $$
\end{utv}

From this bound and Proposition~\ref{polytopedim} one get the following upper bound on the number of vertices of $\Omega_n^d$.

\begin{utv} \label{polytopeupperbound}
For the number of vertices $V(n,d)$ of the polytope of polystochastic matrices $\Omega_n^d$ we have
$$V(n,d) \leq {n^d - \lfloor \frac{(n-1)^d+1}{2} \rfloor \choose n^d - (n-1)^d} +  {n^d - \lfloor \frac{(n-1)^d+2}{2} \rfloor \choose n^d - (n-1)^d}. $$
\end{utv}

For $d=3$, this bound was proved in \cite{Li2Zhang.vertstoch}.

Since every $d$-dimensional permutation of order $n$  is a vertex of $\Omega_n^d$, the number of vertices $V(n,d)$ is bounded from below by the number of such matrices. In~\cite{Keevash.existdesII} it was found the lower bound on  the number of multidimensional permutations, and with the upper bound from~\cite{LinLur.hdimper} it gives the following result.

\begin{teorema}[\cite{Keevash.existdesII}, \cite{LinLur.hdimper}]
The number of $d$-dimensional permutations of order $n$ is 
$\left( \frac{n}{e^{d-1}}  + o(n)\right)^{n^{d-1}} $. 
\end{teorema}

Most constructions and lower bounds on the number of vertices of $\Omega_n^d$ are usually based on matrices with exactly two $\nicefrac{1}{2}$-entries in each line~\cite{CuiLiNg.Birkfortensor, LinLur.birvert, FichSwart.3dimstoch}.
 In~\cite{LinLur.birvert} Linial and Luria estimated the number of vertices of $\Omega_n^3$ from below by $M(n)^{3/2 - o(1) }$, where $M(n)$ is the number of multidimensional permutations.

Summarizing all of these results, we deduce the following asymptotic bounds for the logarithm of $V(n,d)$ when $n \rightarrow \infty$.
\begin{utv}
If $d \geq 4$ is fixed, then for the number $V(n,d)$ of vertices of $\Omega_n^d$ we have
$$ n^{d-1} \ln n (1 + o(1)) \leq \ln  V(n,d) \leq  dn^{d-1} \ln n (1 + o(1)) .$$
If $d = 3$, then 
$$ \frac{3}{2} n^{2}  \ln n(1 + o(1)) \leq \ln V(n,3) \leq  3 n^{2} \ln n (1 + o(1)), $$
and $\ln V(n,2) = \ln  n! = n \ln n (1 + o(1))$. 
\end{utv}

The following lower bounds on the number of  $d$-dimensional permutations of order $n$, when $n$ is fixed, can be found in~\cite{PotKrot.numbernary}.

\begin{teorema}
If $n= 4$, then
$$\log_2 V(4,d) \geq  2^{d-1} + d \log_2 3  + 1 + o(1). $$
If $n = 5$, then 
$$\log_2 V(5,d) \geq   3^{(d-1)/3}  - 0.072. $$
If $n \geq 6$ is even, then 
$$\log_2 V(n,d) \geq   {\left( \frac{n}{2} \right) }^{d-1}, $$
and if  $n \geq 7$ is odd, then 
$$\log_2 V(n,d) \geq   {\left( \frac{n-3}{2} \right)}^{\frac{d-1}{2}} {\left( \frac{n-1}{2} \right)}^{\frac{d-1}{2}}. $$
\end{teorema}

Note that for every $d$ the $d$-dimensional permutation of order $3$ is unique up to the equivalence.

From Proposition~\ref{polytopeupperbound}, we deduce the following:

\begin{utv}
If $n$ is fixed, then  for the logarithm number $V(n,d)$ of vertices of $d$-dimensional polystochastic matrices of order $n$ we have
$$\log_2 V(n,d) \leq \frac{d(n-1)^d}{2}  (\log_2 n - \log_2 (n-1)) + \frac{(n-1)^d}{2} (1 + \log_2 e) + o ((n-1)^d).$$
\end{utv}

A comparison of several lower and upper bounds on the number of vertices of $3$-dimensional polystochastic matrices was also given in~\cite{Zhangx2.extrempoints}.

\section{Constructions of vertices by matrix products}

Let $A$ be a $d$-dimensional matrix of order $n_1$ and $B$ be a $d$-dimensional  matrix of order $n_2$.  Then the \textit{Kronecker product} $A\otimes B$ of matrices $A$ and $B$ is the $d$-dimensional matrix $C$ of order $n_1 n_2$ with entries $c_{\gamma} = a_\alpha b_\beta,$
where  $\gamma_i = (\alpha_i - 1) n_2 + \beta_i$ for each $i = 1, \ldots, d$. 

In~\cite{my.permofproduct} it was proved that the Kronecker product of polystochastic matrices is a polystochastic matrix.

\begin{teorema}
If $A$ is a $d$-dimensional permutation of order $n_1$, and $B$ is a  vertex of $\Omega^d_{n_2}$, then $A \otimes B$ is a vertex of $\Omega^d_{n_1 n_2}$.
\end{teorema}

\begin{proof}
Since $A$ is a permutation matrix, then the Kronecker product $C = A \otimes B$ has  a block structure, where each block is the matrix $B$ or the $d$-dimensional zero matrix of order $n_2$. In other words,  if $a_{\alpha} = 1$, then  $c_{\gamma} = b_\beta,$  where  $\gamma_i = (\alpha_i - 1) n_2 + \beta_i$ for each $i = 1, \ldots, d$, and if $a_{\alpha} = 1$, then  $c_{\gamma} = 0$ for all indices $\beta$ and $\gamma$ such that    $\gamma_i = (\alpha_i - 1) n_2 + \beta_i$ for each $i = 1, \ldots, d$.

If $C = A \otimes B$ is not a vertex of $\Omega_{n_1 n_2}^d$, then there exists a nontrivial   convex combination $C = \lambda_1 D_1  + \lambda_2 D_2$, where  $D_1$ and $D_2 $ are polystochastic matrices that are different from $C$. Then there exists a block $B$ in matrix $C$, for which we have a nontrivial convex combination $B = \lambda_1 D'_1  + \lambda_2 D'_2$. Note that matrices $D'_1$ and $D'_2$ are polystochastic because $C$ is a polystochastic matrix and all other blocks from the same line as a block $B$ in $C$ are zero matrices. Thus, we obtain a nontrivial convex combination for the vertex $B$ of $\Omega_{n_2}^d$: a contradiction.
\end{proof}

\begin{zam}
The Kronecker product of two arbitrary vertices of the Birkhoff polytope is not necessarily a vertex. For example,  the product $V^3_3 \otimes V_3^3$, where $V_3^3 \in \Omega_3^3$, is not a vertex of $\Omega^3_9$. 
\end{zam}

 Let $A$  be a $d_1$-dimensional matrix of order $n$  and $B$ be  a $d_2$-dimensional matrix of order $n$. We define the \textit{dot product} $A \cdot B$ of matrices $A$ and $B$ to be the $(d_1 + d_2 - 2)$-dimensional matrix  $C$  such  that $c_\gamma = \sum\limits_{i=1}^n a_{\alpha i} b_{i \beta}$, where $\gamma = \alpha \beta$ is the concatenation of indices $\alpha \in I_n^{d_1 -1}$ and $\beta \in I_n^{d_2 -1}$.
 
 In~\cite{my.permofproduct}, it was proved that the dot product of polystochastic matrices is a polystochastic matrix.
 
 \begin{lemma} \label{planeofdot}
Let  $A$ be a $d_1$-dimensional permutation matrix of order $n$ and $B$ be a $d_2$-dimensional matrix of order $n$. Then every $d_2$-dimensional plane of the dot product $A \cdot B$ is a matrix equivalent to the matrix $B$.  
 \end{lemma}
 
 \begin{proof}
 The lemma can be easily proved by induction on $d_1$ and by the definition of the dot product.
 \end{proof}

\begin{teorema} \label{dotconstruction}
If $A$ is a $d_1$-dimensional permutation of order $n$, and $B$ is a  vertex of $\Omega^{d_2}_{n}$, then $A \cdot B$ is a vertex of $\Omega^{d_1 + d_2 -2}_{n}$.
\end{teorema}

\begin{proof}
If $C = A \cdot B$ is not a vertex of $\Omega^{d_1 + d_2 -2}_{n}$, then there exists a nontrivial   convex combination $C = \lambda_1 D_1  + \lambda_2 D_2$, where  $D_1$ and $D_2 $ are polystochastic matrices that are different from $C$.  By Lemma~\ref{planeofdot}, there is a $d_2$-dimensional plane $\Gamma $ of  $C$ that is equivalent to the matrix $B$. Note that in this case we have a nontrivial convex combination $B = \lambda_1 D'_1  + \lambda_2 D'_2$, where matrices $D_1$ and $D_2$ are polystochastic: a contradiction.
\end{proof}

\begin{zam}
The dot product of two arbitrary vertices of the Birkhoff polytope is not necessarily  a vertex. For example, the dot product $V^3_3 \cdot V_3^3$  of the vertex  $V_3^3 \in \Omega_3^3$  is not a vertex of $\Omega^4_3$. 
\end{zam}

\section*{Appendix. Vertices of the polytope $\Omega_3^d$}

In this section we list all 21 vertices of $\Omega_3^4$ providing one representative for each equivalence class and  calculate their permanents. Vertices of  $\Omega_3^4$  are exhausted by a computer search of possible supports. By Proposition~\ref{supportbound}, the size of the support of all vertices $\Omega_3^d$ of  is not greater than $65$, and the support of minimal size has a multidimensional permutation $M_3^4$.   Nonexistence of a  nontrivial convex combinations for each vertex is checked by the means of Proposition~\ref{vertexcrit}.

1. The multidimensional permutation:

$$
M_3^4 = \left( 
\begin{array}{ccc|ccc|ccc}
1 & 0 & 0 & 0 & 1 & 0 & 0 & 0 & 1 \\
0 & 1 & 0 & 0 & 0 & 1 & 1 & 0 & 0 \\
0 & 0 & 1 & 1 & 0 & 0 & 0 & 1 & 0 \\
\hline
0 & 1 & 0 & 0 & 0 & 1 & 1 & 0 & 0 \\
0 & 0 & 1 & 1 & 0 & 0 & 0 & 1 & 0 \\
1 & 0 & 0 & 0 & 1 & 0 & 0 & 0 & 1 \\
\hline
0 & 0 & 1 & 1 & 0 & 0 & 0 & 1 & 0 \\
1 & 0 & 0 & 0 & 1 & 0 & 0 & 0 & 1 \\
0 & 1 & 0 & 0 & 0 & 1 & 1 & 0 & 0 \\
\end{array}
\right)
$$
The size of its  support is $N(M_3^4) = 27$ and the permanent $\per (M_3^4) = 27$. 

2. 
$$
A_1 = \left( 
\begin{array}{ccc|ccc|ccc}
1 & 0 & 0 & 0 & 1 & 0 & 0 & 0 & 1 \\
0 & \nicefrac{1}{2} & \nicefrac{1}{2} & \nicefrac{1}{2} & 0 & \nicefrac{1}{2} & \nicefrac{1}{2} &  \nicefrac{1}{2} & 0 \\
0 & \nicefrac{1}{2} & \nicefrac{1}{2} & \nicefrac{1}{2} & 0 & \nicefrac{1}{2} & \nicefrac{1}{2} & \nicefrac{1}{2} & 0 \\
\hline
0 & \nicefrac{1}{2} & \nicefrac{1}{2} & \nicefrac{1}{2} & 0 & \nicefrac{1}{2} & \nicefrac{1}{2} & \nicefrac{1}{2} & 0 \\
1 & 0 & 0 & 0 & \nicefrac{1}{2} & \nicefrac{1}{2} & 0 & \nicefrac{1}{2} & \nicefrac{1}{2} \\
0 & \nicefrac{1}{2} & \nicefrac{1}{2} & \nicefrac{1}{2} & \nicefrac{1}{2} & 0 & \nicefrac{1}{2} & 0 & \nicefrac{1}{2} \\
\hline
0 & \nicefrac{1}{2} & \nicefrac{1}{2} & \nicefrac{1}{2} & 0 & \nicefrac{1}{2} & \nicefrac{1}{2} & \nicefrac{1}{2} & 0 \\
0 & \nicefrac{1}{2} & \nicefrac{1}{2} & \nicefrac{1}{2} & \nicefrac{1}{2} & 0 & \nicefrac{1}{2} & 0 & \nicefrac{1}{2} \\
1 & 0 & 0 & 0 & \nicefrac{1}{2} & \nicefrac{1}{2} & 0 & \nicefrac{1}{2} & \nicefrac{1}{2} \\
\end{array}
\right)
$$
The size of its  support is $N(A_1) = 49$ and the permanent $\per (A_1) = 10.5$.

3. The dot product of the permutation $M_3^3$ and the vertex $V_3^3$ (see Theorem~\ref{dotconstruction}):

$$
M_3^3 \cdot V^3_3 = \left( 
\begin{array}{ccc|ccc|ccc}
1 & 0 & 0 & 0 & \nicefrac{1}{2} & \nicefrac{1}{2} & 0 & \nicefrac{1}{2} & \nicefrac{1}{2} \\
0 & \nicefrac{1}{2} & \nicefrac{1}{2} & \nicefrac{1}{2} & 0 & \nicefrac{1}{2} & \nicefrac{1}{2} &  \nicefrac{1}{2} & 0 \\
0 & \nicefrac{1}{2} & \nicefrac{1}{2} & \nicefrac{1}{2} & \nicefrac{1}{2} & 0 & \nicefrac{1}{2} & 0 & \nicefrac{1}{2} \\
\hline
0 & \nicefrac{1}{2} & \nicefrac{1}{2} & 1 & 0 & 0 & 0 & \nicefrac{1}{2} & \nicefrac{1}{2} \\
\nicefrac{1}{2} & \nicefrac{1}{2} & 0 & 0 & \nicefrac{1}{2} & \nicefrac{1}{2} & \nicefrac{1}{2} & 0 & \nicefrac{1}{2} \\
\nicefrac{1}{2} & 0 & \nicefrac{1}{2} & 0 & \nicefrac{1}{2} & \nicefrac{1}{2} & \nicefrac{1}{2} & \nicefrac{1}{2} & 0\\
\hline
0 & \nicefrac{1}{2} & \nicefrac{1}{2} & 0 & \nicefrac{1}{2} & \nicefrac{1}{2} & 1 & 0 & 0 \\
\nicefrac{1}{2} & 0 & \nicefrac{1}{2} & \nicefrac{1}{2} & \nicefrac{1}{2} & 0 & 0 & \nicefrac{1}{2} & \nicefrac{1}{2} \\
\nicefrac{1}{2} & \nicefrac{1}{2} & 0 & \nicefrac{1}{2} & 0 & \nicefrac{1}{2} & 0 & \nicefrac{1}{2} & \nicefrac{1}{2} \\
\end{array}
\right)
$$
The size of its  support is $N(M_3^3 \cdot V^3_3) = 51$ and the permanent $\per (M_3^3 \cdot V^3_3) = 9$. 

4. 
$$
A_2 = \left( 
\begin{array}{ccc|ccc|ccc}
1 & 0 & 0 & 0 & \nicefrac{1}{2} & \nicefrac{1}{2} & 0 & \nicefrac{1}{2} & \nicefrac{1}{2} \\
0 & \nicefrac{1}{2} & \nicefrac{1}{2} & \nicefrac{1}{2} & 0 & \nicefrac{1}{2} & \nicefrac{1}{2} &  \nicefrac{1}{2} & 0 \\
0 & \nicefrac{1}{2} & \nicefrac{1}{2} & \nicefrac{1}{2} & \nicefrac{1}{2} & 0 & \nicefrac{1}{2} & 0 & \nicefrac{1}{2} \\
\hline
0 & \nicefrac{1}{2} & \nicefrac{1}{2} & \nicefrac{1}{2} & 0 & \nicefrac{1}{2} & \nicefrac{1}{2} &  \nicefrac{1}{2} & 0 \\
\nicefrac{1}{2} & 0 & \nicefrac{1}{2} & 0 & 1 & 0 & \nicefrac{1}{2} & 0 & \nicefrac{1}{2} \\
\nicefrac{1}{2} & \nicefrac{1}{2} & 0 & \nicefrac{1}{2} & 0 & \nicefrac{1}{2} & 0 & \nicefrac{1}{2} & \nicefrac{1}{2} \\
\hline
0 & \nicefrac{1}{2} & \nicefrac{1}{2} & \nicefrac{1}{2} & \nicefrac{1}{2} & 0 & \nicefrac{1}{2} & 0 & \nicefrac{1}{2} \\
\nicefrac{1}{2} & \nicefrac{1}{2} & 0 & \nicefrac{1}{2} & 0 & \nicefrac{1}{2} & 0 & \nicefrac{1}{2} & \nicefrac{1}{2} \\
\nicefrac{1}{2} & 0 & \nicefrac{1}{2} & 0 & \nicefrac{1}{2} & \nicefrac{1}{2} & \nicefrac{1}{2} & \nicefrac{1}{2} & 0 \\
\end{array}
\right)
$$
The size of its  support is $N(A_2) = 52$ and the permanent $\per (A_2) = 8.25$. 

5. 
$$
A_3 = \left( 
\begin{array}{ccc|ccc|ccc}
1 & 0 & 0 & 0 & \nicefrac{1}{3} & \nicefrac{2}{3} & 0 & \nicefrac{2}{3} & \nicefrac{1}{3} \\
0 & \nicefrac{2}{3} & \nicefrac{1}{3} & \nicefrac{2}{3} & \nicefrac{1}{3} & 0 & \nicefrac{1}{3} &  0 & \nicefrac{2}{3} \\
0 & \nicefrac{1}{3} & \nicefrac{2}{3} & \nicefrac{1}{3} & \nicefrac{1}{3} & \nicefrac{1}{3} & \nicefrac{2}{3} & \nicefrac{1}{3} & 0 \\
\hline
0 & \nicefrac{2}{3} & \nicefrac{1}{3} & \nicefrac{2}{3} & \nicefrac{1}{3} & 0 & \nicefrac{1}{3} &  0 & \nicefrac{2}{3} \\
\nicefrac{1}{3} & \nicefrac{1}{3} & \nicefrac{1}{3} & \nicefrac{1}{3} & 0 & \nicefrac{2}{3} & \nicefrac{1}{3} & \nicefrac{2}{3} & 0 \\
\nicefrac{2}{3} & 0 & \nicefrac{1}{3} & 0 & \nicefrac{2}{3} & \nicefrac{1}{3} & \nicefrac{1}{3} & \nicefrac{1}{3} & \nicefrac{1}{3} \\
\hline
0 & \nicefrac{1}{3} & \nicefrac{2}{3} & \nicefrac{1}{3} & \nicefrac{1}{3} & \nicefrac{1}{3} & \nicefrac{2}{3} & \nicefrac{1}{3} & 0 \\
\nicefrac{2}{3} & 0 & \nicefrac{1}{3} & 0 & \nicefrac{2}{3} & \nicefrac{1}{3} & \nicefrac{1}{3} & \nicefrac{1}{3} & \nicefrac{1}{3} \\
\nicefrac{1}{3} & \nicefrac{2}{3} & 0 & \nicefrac{2}{3} & 0 & \nicefrac{1}{3} & 0 & \nicefrac{1}{3} & \nicefrac{2}{3} \\
\end{array}
\right)
$$
The size of its  support is $N(A_3) = 58$ and the permanent $\per (A_3) =  \frac{256}{27} \approx 9.48$.

6. 
$$
A_4 = \left( 
\begin{array}{ccc|ccc|ccc}
1 & 0 & 0 & 0 & \nicefrac{2}{3} & \nicefrac{1}{3} & 0 & \nicefrac{1}{3} & \nicefrac{2}{3} \\
0 & \nicefrac{2}{3} & \nicefrac{1}{3} & \nicefrac{1}{3} & 0 & \nicefrac{2}{3} & \nicefrac{2}{3} &  \nicefrac{1}{3} & 0 \\
0 & \nicefrac{1}{3} & \nicefrac{2}{3} & \nicefrac{2}{3} & \nicefrac{1}{3} & 0 & \nicefrac{1}{3} & \nicefrac{1}{3} & \nicefrac{1}{3} \\
\hline
0 & \nicefrac{2}{3} & \nicefrac{1}{3} & \nicefrac{2}{3} & 0 & \nicefrac{1}{3} & \nicefrac{1}{3} &  \nicefrac{1}{3} & \nicefrac{1}{3} \\
\nicefrac{2}{3} & 0 & \nicefrac{1}{3} & \nicefrac{1}{3} & \nicefrac{1}{3} & \nicefrac{1}{3} & 0 & \nicefrac{2}{3} & \nicefrac{1}{3} \\
\nicefrac{1}{3} & \nicefrac{1}{3} & \nicefrac{1}{3} & 0 & \nicefrac{2}{3} & \nicefrac{1}{3} & \nicefrac{2}{3} & 0 & \nicefrac{1}{3} \\
\hline
0 & \nicefrac{1}{3} & \nicefrac{2}{3} & \nicefrac{1}{3} & \nicefrac{1}{3} & \nicefrac{1}{3} & \nicefrac{2}{3} & \nicefrac{1}{3} & 0 \\
\nicefrac{1}{3} & \nicefrac{1}{3} & \nicefrac{1}{3} & \nicefrac{1}{3} & \nicefrac{2}{3} & 0 & \nicefrac{1}{3} & 0 & \nicefrac{2}{3} \\
\nicefrac{2}{3} & \nicefrac{1}{3} & 0 & \nicefrac{1}{3} & 0 & \nicefrac{2}{3} & 0 & \nicefrac{2}{3} & \nicefrac{1}{3} \\
\end{array}
\right)
$$
The size of its  support is $N(A_4) = 59$ and the permanent $\per (A_4) =  \frac{85}{9} \approx 9.44$.

7. 
$$
A_5 = \left( 
\begin{array}{ccc|ccc|ccc}
1 & 0 & 0 & 0 & \nicefrac{2}{3} & \nicefrac{1}{3} & 0 & \nicefrac{1}{3} & \nicefrac{2}{3} \\
0 & \nicefrac{2}{3} & \nicefrac{1}{3} & \nicefrac{2}{3} & 0 & \nicefrac{1}{3} & \nicefrac{1}{3} &  \nicefrac{1}{3} & \nicefrac{1}{3} \\
0 & \nicefrac{1}{3} & \nicefrac{2}{3} & \nicefrac{1}{3} & \nicefrac{1}{3} & \nicefrac{1}{3} & \nicefrac{2}{3} & \nicefrac{1}{3} & 0 \\
\hline
0 & \nicefrac{2}{3} & \nicefrac{1}{3} & \nicefrac{2}{3} & 0 & \nicefrac{1}{3} & \nicefrac{1}{3} &  \nicefrac{1}{3} & \nicefrac{1}{3} \\
\nicefrac{2}{3} & 0 & \nicefrac{1}{3} & 0 & \nicefrac{1}{3} & \nicefrac{2}{3} & \nicefrac{1}{3} & \nicefrac{2}{3} & 0 \\
\nicefrac{1}{3} & \nicefrac{1}{3} & \nicefrac{1}{3} & \nicefrac{1}{3} & \nicefrac{2}{3} & 0 & \nicefrac{1}{3} & 0 & \nicefrac{2}{3} \\
\hline
0 & \nicefrac{1}{3} & \nicefrac{2}{3} & \nicefrac{1}{3} & \nicefrac{1}{3} & \nicefrac{1}{3} & \nicefrac{2}{3} & \nicefrac{1}{3} & 0 \\
\nicefrac{1}{3} & \nicefrac{1}{3} & \nicefrac{1}{3} & \nicefrac{1}{3} & \nicefrac{2}{3} & 0 & \nicefrac{1}{3} & 0 & \nicefrac{2}{3} \\
\nicefrac{2}{3} & \nicefrac{1}{3} & 0 & \nicefrac{1}{3} & 0 & \nicefrac{2}{3} & 0 & \nicefrac{2}{3} & \nicefrac{1}{3} \\
\end{array}
\right)
$$
The size of its  support is $N(A_5) = 59$ and the permanent $\per (A_5) =  \frac{85}{9} \approx 9.44$.

8. 
$$
A_6 = \left( 
\begin{array}{ccc|ccc|ccc}
1 & 0 & 0 & 0 & \nicefrac{1}{2} & \nicefrac{1}{2} & 0 & \nicefrac{1}{2} & \nicefrac{1}{2} \\
0 & \nicefrac{1}{2} & \nicefrac{1}{2} & \nicefrac{1}{2} & 0 & \nicefrac{1}{2} & \nicefrac{1}{2} &  \nicefrac{1}{2} & 0 \\
0 & \nicefrac{1}{2} & \nicefrac{1}{2} & \nicefrac{1}{2} & \nicefrac{1}{2} & 0 & \nicefrac{1}{2} & 0 & \nicefrac{1}{2} \\
\hline
0 & \nicefrac{3}{4} & \nicefrac{1}{4} & \nicefrac{3}{4} & 0 & \nicefrac{1}{4} & \nicefrac{1}{4} &  \nicefrac{1}{4} & \nicefrac{1}{2} \\
\nicefrac{3}{4} & 0 & \nicefrac{1}{4} & 0 & \nicefrac{1}{2} & \nicefrac{1}{2} & \nicefrac{1}{4} & \nicefrac{1}{2} & \nicefrac{1}{4} \\
\nicefrac{1}{4} & \nicefrac{1}{4} & \nicefrac{1}{2} & \nicefrac{1}{4} & \nicefrac{1}{2} & \nicefrac{1}{4} & \nicefrac{1}{2} & \nicefrac{1}{4} & \nicefrac{1}{4} \\
\hline
0 & \nicefrac{1}{4} & \nicefrac{3}{4} & \nicefrac{1}{4} & \nicefrac{1}{2} & \nicefrac{1}{4} & \nicefrac{3}{4} & \nicefrac{1}{4} & 0 \\
\nicefrac{1}{4} & \nicefrac{1}{2} & \nicefrac{1}{4} & \nicefrac{1}{2} & \nicefrac{1}{2} & 0 & \nicefrac{1}{4} & 0 & \nicefrac{3}{4} \\
\nicefrac{3}{4} & \nicefrac{1}{4} & 0 & \nicefrac{1}{4} & 0 & \nicefrac{3}{4} & 0 & \nicefrac{3}{4} & \nicefrac{1}{4} \\
\end{array}
\right)
$$
The size of its  support is $N(A_6) = 60$ and the permanent $\per (A_6) =  \frac{77}{8} = 9.625$.

9. 
$$
A_7 = \left( 
\begin{array}{ccc|ccc|ccc}
0 & \nicefrac{1}{2} & \nicefrac{1}{2} & \nicefrac{1}{2} & 0 & \nicefrac{1}{2} & \nicefrac{1}{2} & \nicefrac{1}{2} & 0 \\
\nicefrac{1}{2} & 0 & \nicefrac{1}{2} & 0 & \nicefrac{3}{4} & \nicefrac{1}{4} & \nicefrac{1}{2} &  \nicefrac{1}{4} & \nicefrac{1}{4} \\
\nicefrac{1}{2} & \nicefrac{1}{2} & 0 & \nicefrac{1}{2} & \nicefrac{1}{4} & \nicefrac{1}{4} & 0 & \nicefrac{1}{4} & \nicefrac{3}{4} \\
\hline
\nicefrac{1}{2} & 0 & \nicefrac{1}{2} & 0 & \nicefrac{3}{4} & \nicefrac{1}{4} & \nicefrac{1}{2} &  \nicefrac{1}{4} & \nicefrac{1}{4} \\
0 & \nicefrac{3}{4} & \nicefrac{1}{4} & \nicefrac{3}{4} & \nicefrac{1}{4} & 0 & \nicefrac{1}{4} & 0 & \nicefrac{1}{4} \\
\nicefrac{1}{2} & \nicefrac{1}{4} & \nicefrac{1}{4} & \nicefrac{1}{4} & 0 & \nicefrac{3}{4} & \nicefrac{1}{4} & \nicefrac{3}{4} & 0 \\
\hline
\nicefrac{1}{2} & \nicefrac{1}{2} & 0 & \nicefrac{1}{2} & \nicefrac{1}{4} & \nicefrac{1}{4} & 0 & \nicefrac{1}{4} & \nicefrac{3}{4} \\
\nicefrac{1}{2} & \nicefrac{1}{4} & \nicefrac{1}{4} & \nicefrac{1}{4} & 0 & \nicefrac{3}{4} & \nicefrac{1}{4} & \nicefrac{3}{4} & 0 \\
0 & \nicefrac{1}{4} & \nicefrac{3}{4} & \nicefrac{1}{4} & \nicefrac{3}{4} & 0 & \nicefrac{3}{4} & 0 & \nicefrac{1}{4} \\
\end{array}
\right)
$$
The size of its  support is $N(A_7) = 60$ and the permanent $\per (A_7) =  \frac{39}{4} = 9.75$.

10. 
$$
A_8 = \left( 
\begin{array}{ccc|ccc|ccc}
\nicefrac{1}{3} & \nicefrac{1}{3} & \nicefrac{1}{3} & \nicefrac{1}{3} & 0 & \nicefrac{2}{3} & \nicefrac{1}{3} & \nicefrac{2}{3} & 0 \\
\nicefrac{1}{3} & \nicefrac{1}{3} & \nicefrac{1}{3} & \nicefrac{2}{3} & \nicefrac{1}{3} & 0 & 0 &  \nicefrac{1}{3} & \nicefrac{2}{3} \\
\nicefrac{1}{3} & \nicefrac{1}{3} & \nicefrac{1}{3} & 0 & \nicefrac{2}{3} & \nicefrac{1}{3} & \nicefrac{2}{3} & 0 & \nicefrac{1}{3} \\
\hline
\nicefrac{2}{3} & 0 & \nicefrac{1}{3} & \nicefrac{1}{3} & \nicefrac{2}{3} & 0 & 0 &  \nicefrac{1}{3} & \nicefrac{2}{3} \\
0 & \nicefrac{1}{3} & \nicefrac{2}{3} & \nicefrac{1}{3} & \nicefrac{1}{3} & \nicefrac{1}{3} & \nicefrac{2}{3} & \nicefrac{1}{3} & 0 \\
\nicefrac{1}{3} & \nicefrac{2}{3} & 0 & \nicefrac{1}{3} & 0 & \nicefrac{2}{3} & \nicefrac{1}{3} & \nicefrac{1}{3} & \nicefrac{1}{3} \\
\hline
0 & \nicefrac{2}{3} & \nicefrac{1}{3} & \nicefrac{1}{3} & \nicefrac{1}{3} & \nicefrac{1}{3} & \nicefrac{2}{3} & 0 & \nicefrac{1}{3} \\
\nicefrac{2}{3} & \nicefrac{1}{3} & 0  & 0 & \nicefrac{1}{3} & \nicefrac{2}{3} & \nicefrac{1}{3} & \nicefrac{1}{3} & \nicefrac{1}{3} \\
\nicefrac{1}{3} & 0 & \nicefrac{2}{3} & \nicefrac{2}{3} & \nicefrac{1}{3} & 0 & 0 & \nicefrac{2}{3} & \nicefrac{1}{3} \\
\end{array}
\right)
$$
The size of its  support is $N(A_{8}) = 61$ and the permanent $\per (A_8) =  \frac{236}{27} \approx 8.74$.  This vertex of $\Omega_3^4$  was also found in~\cite{AhLoHem.polycones}.

11. 
$$
A_9 = \left( 
\begin{array}{ccc|ccc|ccc}
\nicefrac{1}{3} & \nicefrac{1}{3} & \nicefrac{1}{3} & \nicefrac{1}{3} & 0 & \nicefrac{2}{3} & \nicefrac{1}{3} & \nicefrac{2}{3} & 0 \\
\nicefrac{1}{3} & \nicefrac{1}{3} & \nicefrac{1}{3} & \nicefrac{2}{3} & \nicefrac{1}{3} & 0 & 0 &  \nicefrac{1}{3} & \nicefrac{2}{3} \\
\nicefrac{1}{3} & \nicefrac{1}{3} & \nicefrac{1}{3} & 0 & \nicefrac{2}{3} & \nicefrac{1}{3} & \nicefrac{2}{3} & 0 & \nicefrac{1}{3} \\
\hline
\nicefrac{1}{3} & \nicefrac{1}{3} & \nicefrac{1}{3} & 0 & \nicefrac{2}{3} & \nicefrac{1}{3} & \nicefrac{2}{3} &  0 & \nicefrac{1}{3} \\
\nicefrac{1}{3} & 0 & \nicefrac{2}{3} & \nicefrac{1}{3} & \nicefrac{1}{3} & \nicefrac{1}{3} & \nicefrac{1}{3} & \nicefrac{2}{3} & 0 \\
\nicefrac{1}{3} & \nicefrac{2}{3} & 0 & \nicefrac{2}{3} & 0 & \nicefrac{1}{3} & 0 & \nicefrac{1}{3} & \nicefrac{2}{3} \\
\hline
\nicefrac{1}{3} & \nicefrac{1}{3} & \nicefrac{1}{3} & \nicefrac{2}{3} & \nicefrac{1}{3} & 0 & 0 & \nicefrac{1}{3} & \nicefrac{2}{3} \\
\nicefrac{1}{3} & \nicefrac{2}{3} & 0  & 0 & \nicefrac{1}{3} & \nicefrac{2}{3} & \nicefrac{2}{3} & 0 & \nicefrac{1}{3} \\
\nicefrac{1}{3} & 0 & \nicefrac{2}{3} & \nicefrac{1}{3} & \nicefrac{1}{3} & \nicefrac{1}{3} & \nicefrac{1}{3} & \nicefrac{2}{3} & 0 \\
\end{array}
\right)
$$
The size of its  support is $N(A_{9}) = 61$ and the permanent $\per (A_9) =  \frac{236}{27} \approx 8.74$.

12. 
$$
A_{10} = \left( 
\begin{array}{ccc|ccc|ccc}
1 & 0 & 0 & 0 & \nicefrac{1}{2} & \nicefrac{1}{2} & 0 & \nicefrac{1}{2} & \nicefrac{1}{2} \\
0 & \nicefrac{1}{2} & \nicefrac{1}{2} & \nicefrac{1}{2} & 0 & \nicefrac{1}{2} & \nicefrac{1}{2} &  \nicefrac{1}{2} & 0 \\
0 & \nicefrac{1}{2} & \nicefrac{1}{2} & \nicefrac{1}{2} & \nicefrac{1}{2} & 0 & \nicefrac{1}{2} & 0 & \nicefrac{1}{2} \\
\hline
0 & \nicefrac{1}{4} & \nicefrac{3}{4} & \nicefrac{3}{4} & \nicefrac{1}{4} & 0 & \nicefrac{1}{4} &  \nicefrac{1}{2} & \nicefrac{1}{4} \\
\nicefrac{3}{4} & \nicefrac{1}{4} & 0 & 0 & \nicefrac{1}{2} & \nicefrac{1}{2} & \nicefrac{1}{4} & \nicefrac{1}{4} & \nicefrac{1}{2} \\
\nicefrac{1}{4} & \nicefrac{1}{2} & \nicefrac{1}{4} & \nicefrac{1}{4} & \nicefrac{1}{4} & \nicefrac{1}{2} & \nicefrac{1}{2} & \nicefrac{1}{4} & \nicefrac{1}{4} \\
\hline
0 & \nicefrac{3}{4} & \nicefrac{1}{4} & \nicefrac{1}{4} & \nicefrac{1}{4} & \nicefrac{1}{2} & \nicefrac{3}{4} & 0 & \nicefrac{1}{4} \\
\nicefrac{1}{4} & \nicefrac{1}{4} & \nicefrac{1}{2}  & \nicefrac{1}{2} & \nicefrac{1}{2} & 0 & \nicefrac{1}{4} & \nicefrac{1}{4} & \nicefrac{1}{2} \\
\nicefrac{3}{4} & 0 & \nicefrac{1}{4} & \nicefrac{1}{4} & \nicefrac{1}{4} & \nicefrac{1}{2} & 0 & \nicefrac{3}{4} & \nicefrac{1}{4} \\
\end{array}
\right)
$$
The size of its  support is $N(A_{10}) = 62$ and the permanent $\per (A_{10}) =  \frac{75}{8} =  9.375$.

13. 
$$
A_{11} = \left( 
\begin{array}{ccc|ccc|ccc}
1 & 0 & 0 & 0 & \nicefrac{3}{5} & \nicefrac{2}{5} & 0 & \nicefrac{2}{5} & \nicefrac{3}{5} \\
0 & \nicefrac{3}{5} & \nicefrac{2}{5} & \nicefrac{3}{5} & 0 & \nicefrac{2}{5} & \nicefrac{2}{5} &  \nicefrac{2}{5} & \nicefrac{1}{5} \\
0 & \nicefrac{2}{5} & \nicefrac{3}{5} & \nicefrac{2}{5} & \nicefrac{2}{5} & \nicefrac{1}{5} & \nicefrac{3}{5} & \nicefrac{1}{5} & \nicefrac{1}{5} \\
\hline
0 & \nicefrac{3}{5} & \nicefrac{2}{5} & \nicefrac{3}{5} & 0 & \nicefrac{2}{5} & \nicefrac{2}{5} &  \nicefrac{2}{5} & \nicefrac{1}{5} \\
\nicefrac{3}{5} & 0 & \nicefrac{2}{5} & 0 & \nicefrac{2}{5} & \nicefrac{3}{5} & \nicefrac{2}{5} & \nicefrac{3}{5} & 0 \\
\nicefrac{2}{5} & \nicefrac{2}{5} & \nicefrac{1}{5} & \nicefrac{2}{5} & \nicefrac{3}{5} & 0 & \nicefrac{1}{5} & 0 & \nicefrac{4}{5} \\
\hline
0 & \nicefrac{2}{5} & \nicefrac{3}{5} & \nicefrac{2}{5} & \nicefrac{2}{5} & \nicefrac{1}{5} & \nicefrac{3}{5} & \nicefrac{1}{5} & \nicefrac{1}{5} \\
\nicefrac{2}{5} & \nicefrac{2}{5} & \nicefrac{1}{5}  & \nicefrac{2}{5} & \nicefrac{3}{5} & 0 & \nicefrac{1}{5} & 0 & \nicefrac{4}{5} \\
\nicefrac{3}{5} & \nicefrac{1}{5} & \nicefrac{1}{5} & \nicefrac{1}{5} & 0 & \nicefrac{4}{5} & \nicefrac{1}{5} & \nicefrac{4}{5} & 0 \\
\end{array}
\right)
$$
The size of its  support is $N(A_{11}) = 62$ and the permanent $\per (A_{11}) =  \frac{1194}{125} =  9.552$.

14. 
$$
A_{12} = \left( 
\begin{array}{ccc|ccc|ccc}
\nicefrac{1}{2} & \nicefrac{1}{4} & \nicefrac{1}{4} & \nicefrac{1}{2} & 0 & \nicefrac{1}{2} & 0 & \nicefrac{3}{4} & \nicefrac{1}{4} \\
\nicefrac{1}{4} & \nicefrac{1}{2} & \nicefrac{1}{4} & \nicefrac{1}{2} & \nicefrac{1}{2} & 0 & \nicefrac{1}{4} & 0 & \nicefrac{3}{4} \\
\nicefrac{1}{4} & \nicefrac{1}{4} & \nicefrac{1}{2} & 0 & \nicefrac{1}{2} & \nicefrac{1}{2} & \nicefrac{3}{4} & \nicefrac{1}{4} & 0 \\
\hline
\nicefrac{1}{2} & \nicefrac{1}{2} & 0 & \nicefrac{1}{4} & \nicefrac{1}{4} & \nicefrac{1}{2} & \nicefrac{1}{4} &  \nicefrac{1}{4} & \nicefrac{1}{2} \\
0 & \nicefrac{1}{2} & \nicefrac{1}{2} & \nicefrac{1}{4} & \nicefrac{1}{4} & \nicefrac{1}{2} & \nicefrac{3}{4} & \nicefrac{1}{4} & 0 \\
\nicefrac{1}{2} & 0 & \nicefrac{1}{2} & \nicefrac{1}{2} & \nicefrac{1}{2} & 0 & 0 & \nicefrac{1}{2} & \nicefrac{1}{2} \\
\hline
0 & \nicefrac{1}{4} & \nicefrac{3}{4} & \nicefrac{1}{4} & \nicefrac{3}{4} & 0 & \nicefrac{3}{4} & 0 & \nicefrac{1}{4} \\
\nicefrac{3}{4} & 0 & \nicefrac{1}{4}  & \nicefrac{1}{4} & \nicefrac{1}{4} & \nicefrac{1}{2} & 0 & \nicefrac{3}{4} & \nicefrac{1}{4} \\
\nicefrac{1}{4} & \nicefrac{3}{4} & 0 & \nicefrac{1}{2} & 0 & \nicefrac{1}{2} & \nicefrac{1}{4} & \nicefrac{1}{4} & \nicefrac{1}{2} \\
\end{array}
\right)
$$
The size of its  support is $N(A_{12}) = 62$ and the permanent $\per (A_{12}) =  \frac{73}{8} =  9.125$.

15. 
$$
A_{13} = \left( 
\begin{array}{ccc|ccc|ccc}
\nicefrac{1}{3} & \nicefrac{1}{3} & \nicefrac{1}{3} & \nicefrac{1}{3} & 0 & \nicefrac{2}{3} & \nicefrac{1}{3} & \nicefrac{2}{3} & 0 \\
\nicefrac{1}{3} & \nicefrac{1}{3} & \nicefrac{1}{3} & 0 & \nicefrac{2}{3} & \nicefrac{1}{3} & \nicefrac{2}{3} & 0 & \nicefrac{1}{3} \\
\nicefrac{1}{3} & \nicefrac{1}{3} & \nicefrac{1}{3} & \nicefrac{2}{3} & \nicefrac{1}{3} & 0 & 0 & \nicefrac{1}{3} & \nicefrac{2}{3} \\
\hline
\nicefrac{1}{3} & \nicefrac{1}{3} & \nicefrac{1}{3} & \nicefrac{1}{3} & \nicefrac{1}{3} & \nicefrac{1}{3} & \nicefrac{1}{3} &  \nicefrac{1}{3} & \nicefrac{1}{3} \\
\nicefrac{1}{3} & 0 & \nicefrac{2}{3} & \nicefrac{2}{3} & \nicefrac{1}{3} & 0 & 0 & \nicefrac{2}{3} & \nicefrac{1}{3} \\
\nicefrac{1}{3} & \nicefrac{2}{3} & 0 & 0 & \nicefrac{1}{3} & \nicefrac{2}{3} & \nicefrac{2}{3} & 0 & \nicefrac{1}{3} \\
\hline
\nicefrac{1}{3} & \nicefrac{1}{3} & \nicefrac{1}{3} & \nicefrac{1}{3} & \nicefrac{2}{3} & 0 & \nicefrac{1}{3} & 0 & \nicefrac{2}{3} \\
\nicefrac{1}{3} & \nicefrac{2}{3} & 0  & \nicefrac{1}{3} & 0 & \nicefrac{2}{3} & \nicefrac{1}{3} & \nicefrac{1}{3} & \nicefrac{1}{3} \\
\nicefrac{1}{3} & 0 & \nicefrac{2}{3} & \nicefrac{1}{3} & \nicefrac{1}{3} & \nicefrac{1}{3} & \nicefrac{1}{3} & \nicefrac{2}{3} & 0 \\
\end{array}
\right)
$$
The size of its  support is $N(A_{13}) = 63$ and the permanent $\per (A_{13}) =  \frac{26}{3} \approx  8.667$.

16. 
$$
A_{14} = \left( 
\begin{array}{ccc|ccc|ccc}
\nicefrac{1}{2} & \nicefrac{1}{4} & \nicefrac{1}{4} & \nicefrac{1}{2} & 0 & \nicefrac{1}{2} & 0 & \nicefrac{3}{4} & \nicefrac{1}{4} \\
\nicefrac{1}{4} & \nicefrac{1}{2} & \nicefrac{1}{4} & \nicefrac{1}{2} & \nicefrac{1}{2} & 0 & \nicefrac{1}{4} & 0 & \nicefrac{3}{4} \\
\nicefrac{1}{4} & \nicefrac{1}{4} & \nicefrac{1}{2} & 0 & \nicefrac{1}{2} & \nicefrac{1}{2} & \nicefrac{3}{4} & \nicefrac{1}{4} & 0 \\
\hline
\nicefrac{1}{2} & 0 & \nicefrac{1}{2} & \nicefrac{1}{4} & \nicefrac{3}{4} & 0 & \nicefrac{1}{4} &  \nicefrac{1}{4} & \nicefrac{1}{2} \\
0 & \nicefrac{1}{2} & \nicefrac{1}{2} & \nicefrac{1}{4} & \nicefrac{1}{4} & \nicefrac{1}{2} & \nicefrac{3}{4} & \nicefrac{1}{4} & 0 \\
\nicefrac{1}{2} & \nicefrac{1}{2} & 0 & \nicefrac{1}{2} & 0 & \nicefrac{1}{2} & 0 & \nicefrac{1}{2} & \nicefrac{1}{2} \\
\hline
0 & \nicefrac{3}{4} & \nicefrac{1}{4} & \nicefrac{1}{4} & \nicefrac{1}{4} & \nicefrac{1}{2} & \nicefrac{3}{4} & 0 & \nicefrac{1}{4} \\
\nicefrac{3}{4} & 0 & \nicefrac{1}{4}  & \nicefrac{1}{4} & \nicefrac{1}{4} & \nicefrac{1}{2} & 0 & \nicefrac{3}{4} & \nicefrac{1}{4} \\
\nicefrac{1}{4} & \nicefrac{1}{4} & \nicefrac{1}{2} & \nicefrac{1}{2} & \nicefrac{1}{2} & 0 & \nicefrac{1}{4} & \nicefrac{1}{4} & \nicefrac{1}{2} \\
\end{array}
\right)
$$
The size of its  support is $N(A_{14}) = 63$ and the permanent $\per (A_{14}) =  9 $.

17. 
$$
A_{15} = \left( 
\begin{array}{ccc|ccc|ccc}
0 & \nicefrac{4}{5} & \nicefrac{1}{5} & \nicefrac{4}{5} & 0 & \nicefrac{1}{5} & \nicefrac{1}{5} & \nicefrac{1}{5} & \nicefrac{3}{5} \\
\nicefrac{4}{5} & 0 & \nicefrac{1}{5} & 0 & \nicefrac{3}{5} & \nicefrac{2}{5} & \nicefrac{1}{5} &  \nicefrac{2}{5} & \nicefrac{2}{5} \\
\nicefrac{1}{5} & \nicefrac{1}{5} & \nicefrac{3}{5} & \nicefrac{1}{5} & \nicefrac{2}{5} & \nicefrac{2}{5} & \nicefrac{3}{5} & \nicefrac{2}{5} & 0 \\
\hline
\nicefrac{3}{5} & \nicefrac{1}{5} & \nicefrac{1}{5} & \nicefrac{1}{5} & \nicefrac{2}{5} & \nicefrac{2}{5} & \nicefrac{1}{5} &  \nicefrac{2}{5} & \nicefrac{2}{5} \\
\nicefrac{1}{5} & \nicefrac{2}{5} & \nicefrac{2}{5} & \nicefrac{2}{5} & 0 & \nicefrac{3}{5} & \nicefrac{2}{5} & \nicefrac{3}{5} & 0 \\
\nicefrac{1}{5} & \nicefrac{2}{5} & \nicefrac{2}{5} & \nicefrac{2}{5} & \nicefrac{3}{5} & 0 & \nicefrac{2}{5} & 0 & \nicefrac{3}{5} \\
\hline
\nicefrac{2}{5} & 0 & \nicefrac{3}{5} & 0 & \nicefrac{3}{5} & \nicefrac{2}{5} & \nicefrac{3}{5} & \nicefrac{2}{5} & 0 \\
0 & \nicefrac{3}{5} & \nicefrac{2}{5}  & \nicefrac{3}{5} & \nicefrac{2}{5} & 0 & \nicefrac{2}{5} & 0 & \nicefrac{3}{5} \\
\nicefrac{3}{5} & \nicefrac{2}{5} & 0 & \nicefrac{2}{5} & 0 & \nicefrac{3}{5} & 0 & \nicefrac{3}{5} & \nicefrac{2}{5} \\
\end{array}
\right)
$$
The size of its  support is $N(A_{15}) = 63$ and the permanent $\per (A_{15}) =  \frac{1074}{125} =  8.592$.

18. 
$$
A_{16} = \left( 
\begin{array}{ccc|ccc|ccc}
0 & \nicefrac{4}{5} & \nicefrac{1}{5} & \nicefrac{4}{5} & 0 & \nicefrac{1}{5} & \nicefrac{1}{5} & \nicefrac{1}{5} & \nicefrac{3}{5} \\
\nicefrac{4}{5} & 0 & \nicefrac{1}{5} & 0 & \nicefrac{3}{5} & \nicefrac{2}{5} & \nicefrac{1}{5} &  \nicefrac{2}{5} & \nicefrac{2}{5} \\
\nicefrac{1}{5} & \nicefrac{1}{5} & \nicefrac{3}{5} & \nicefrac{1}{5} & \nicefrac{2}{5} & \nicefrac{2}{5} & \nicefrac{3}{5} & \nicefrac{2}{5} & 0 \\
\hline
\nicefrac{4}{5} & \nicefrac{1}{5} & 0 & 0 & \nicefrac{2}{5} & \nicefrac{3}{5} & \nicefrac{1}{5} &  \nicefrac{2}{5} & \nicefrac{2}{5} \\
0 & \nicefrac{2}{5} & \nicefrac{3}{5} & \nicefrac{3}{5} & 0 & \nicefrac{2}{5} & \nicefrac{2}{5} & \nicefrac{3}{5} & 0 \\
\nicefrac{1}{5} & \nicefrac{2}{5} & \nicefrac{2}{5} & \nicefrac{2}{5} & \nicefrac{3}{5} & 0 & \nicefrac{2}{5} & 0 & \nicefrac{3}{5} \\
\hline
\nicefrac{1}{5} & 0 & \nicefrac{4}{5} & \nicefrac{1}{5} & \nicefrac{3}{5} & \nicefrac{1}{5} & \nicefrac{3}{5} & \nicefrac{2}{5} & 0 \\
\nicefrac{1}{5} & \nicefrac{3}{5} & \nicefrac{1}{5}  & \nicefrac{1}{5} & \nicefrac{2}{5} & \nicefrac{1}{5} & \nicefrac{2}{5} & 0 & \nicefrac{3}{5} \\
\nicefrac{3}{5} & \nicefrac{2}{5} & 0 & \nicefrac{2}{5} & 0 & \nicefrac{3}{5} & 0 & \nicefrac{3}{5} & \nicefrac{2}{5} \\
\end{array}
\right)
$$
The size of its  support is $N(A_{16}) = 63$ and the permanent $\per (A_{16}) =  \frac{1141}{125} =  9.128$.

19. 
$$
A_{17} = \left( 
\begin{array}{ccc|ccc|ccc}
0 & \nicefrac{5}{6} & \nicefrac{1}{6} & \nicefrac{5}{6} & 0 & \nicefrac{1}{6} & \nicefrac{1}{6} & \nicefrac{1}{6} & \nicefrac{2}{3} \\
\nicefrac{5}{6} & 0 & \nicefrac{1}{6} & 0 & \nicefrac{1}{2} & \nicefrac{1}{2} & \nicefrac{1}{6} &  \nicefrac{1}{2} & \nicefrac{1}{3} \\
\nicefrac{1}{6} & \nicefrac{1}{6} & \nicefrac{2}{3} & \nicefrac{1}{6} & \nicefrac{1}{2} & \nicefrac{1}{3} & \nicefrac{2}{3} & \nicefrac{1}{3} & 0 \\
\hline
\nicefrac{5}{6} & 0 & \nicefrac{1}{6} & 0 & \nicefrac{1}{2} & \nicefrac{1}{2} & \nicefrac{1}{6} &  \nicefrac{1}{2} & \nicefrac{1}{3} \\
0 & \nicefrac{1}{2} & \nicefrac{1}{2} & \nicefrac{1}{2} & 0 & \nicefrac{1}{2} & \nicefrac{1}{2} & \nicefrac{1}{2} & 0 \\
\nicefrac{1}{6} & \nicefrac{1}{2} & \nicefrac{1}{3} & \nicefrac{1}{2} & \nicefrac{1}{2} & 0 & \nicefrac{1}{3} & 0 & \nicefrac{2}{3} \\
\hline
\nicefrac{1}{6} & \nicefrac{1}{6} & \nicefrac{2}{3} & \nicefrac{1}{6} & \nicefrac{1}{2} & \nicefrac{1}{3} & \nicefrac{2}{3} & \nicefrac{1}{3} & 0 \\
\nicefrac{1}{6} & \nicefrac{1}{2} & \nicefrac{1}{3}  & \nicefrac{1}{2} & \nicefrac{1}{2} & 0 & \nicefrac{1}{3} & 0 & \nicefrac{2}{3} \\
\nicefrac{2}{3} & \nicefrac{1}{3} & 0 & \nicefrac{1}{3} & 0 & \nicefrac{2}{3} & 0 & \nicefrac{2}{3} & \nicefrac{1}{3} \\
\end{array}
\right)
$$
The size of its  support is $N(A_{17}) = 63$ and the permanent $\per (A_{17}) =  \frac{85}{9} \approx  9.444$.

20. 
$$
A_{18} = \left( 
\begin{array}{ccc|ccc|ccc}
\nicefrac{1}{4} & \nicefrac{1}{2} & \nicefrac{1}{4} & \nicefrac{1}{2} & \nicefrac{1}{4} & \nicefrac{1}{4} & \nicefrac{1}{4} & \nicefrac{1}{4} & \nicefrac{1}{2} \\
\nicefrac{1}{2} & \nicefrac{1}{4} & \nicefrac{1}{4} & \nicefrac{1}{4} & 0 & \nicefrac{3}{4} & \nicefrac{1}{4} & \nicefrac{3}{4} & 0 \\
\nicefrac{1}{4} & \nicefrac{1}{4} & \nicefrac{1}{2} & \nicefrac{1}{4} & \nicefrac{3}{4} & 0 & \nicefrac{1}{2} & 0 & \nicefrac{1}{2} \\
\hline
\nicefrac{1}{4} & \nicefrac{1}{2} & \nicefrac{1}{4} & 0 & \nicefrac{1}{4} & \nicefrac{3}{4} & \nicefrac{3}{4} &  \nicefrac{1}{4} & 0 \\
\nicefrac{1}{2} & 0 & \nicefrac{1}{2} & \nicefrac{1}{4} & \nicefrac{3}{4} & 0 & \nicefrac{1}{4} & \nicefrac{1}{4} & \nicefrac{1}{2} \\
\nicefrac{1}{4} & \nicefrac{1}{2} & \nicefrac{1}{4} & \nicefrac{3}{4} & 0 & \nicefrac{1}{4} & 0 & \nicefrac{1}{2} & \nicefrac{1}{2} \\
\hline
\nicefrac{1}{2} & 0 & \nicefrac{1}{2} & \nicefrac{1}{2} & \nicefrac{1}{2} & 0 & 0 & \nicefrac{1}{2} & \nicefrac{1}{2} \\
0 & \nicefrac{3}{4} & \nicefrac{1}{4}  & \nicefrac{1}{2} & \nicefrac{1}{4} & \nicefrac{1}{4} & \nicefrac{1}{2} & 0 & \nicefrac{1}{2} \\
\nicefrac{1}{2} & \nicefrac{1}{4} & \nicefrac{1}{4} & 0 & \nicefrac{1}{4} & \nicefrac{3}{4} & \nicefrac{1}{2} & \nicefrac{1}{2} & 0 \\
\end{array}
\right)
$$
The size of its  support is $N(A_{18}) = 64$ and the permanent $\per (A_{18}) =  \frac{145}{16} = 9.0625$.

21. 
$$
A_{19} = \left( 
\begin{array}{ccc|ccc|ccc}
\nicefrac{1}{5} & \nicefrac{2}{5} & \nicefrac{2}{5} & \nicefrac{3}{5} & \nicefrac{1}{5} & \nicefrac{1}{5} & \nicefrac{1}{5} & \nicefrac{2}{5} & \nicefrac{2}{5} \\
\nicefrac{2}{5} & \nicefrac{1}{5} & \nicefrac{2}{5} & \nicefrac{1}{5} & \nicefrac{4}{5} & 0 & \nicefrac{2}{5} &  0 & \nicefrac{3}{5} \\
\nicefrac{2}{5} & \nicefrac{2}{5} & \nicefrac{1}{5} & \nicefrac{1}{5} & 0 & \nicefrac{4}{5} & \nicefrac{2}{5} & \nicefrac{3}{5} & 0 \\
\hline
\nicefrac{2}{5} & \nicefrac{3}{5} & 0 & 0 & \nicefrac{2}{5} & \nicefrac{3}{5} & \nicefrac{3}{5} &  0 & \nicefrac{2}{5} \\
0 & \nicefrac{2}{5} & \nicefrac{3}{5} & \nicefrac{4}{5} & 0 & \nicefrac{1}{5} & \nicefrac{1}{5} & \nicefrac{3}{5} & \nicefrac{1}{5} \\
\nicefrac{3}{5} & 0 & \nicefrac{2}{5} & \nicefrac{1}{5} & \nicefrac{3}{5} & \nicefrac{1}{5} & \nicefrac{1}{5} & \nicefrac{2}{5} & \nicefrac{2}{5} \\
\hline
\nicefrac{2}{5} & 0 & \nicefrac{3}{5} & \nicefrac{2}{5} & \nicefrac{2}{5} & \nicefrac{1}{5} & \nicefrac{1}{5} & \nicefrac{2}{5} & \nicefrac{1}{5} \\
\nicefrac{3}{5} & \nicefrac{2}{5} & 0  & 0 & \nicefrac{1}{5} & \nicefrac{4}{5} & \nicefrac{2}{5} & \nicefrac{2}{5} & \nicefrac{1}{5} \\
0 & \nicefrac{3}{5} & \nicefrac{2}{5} & \nicefrac{3}{5} & \nicefrac{2}{5} & 0 & \nicefrac{2}{5} & 0 & \nicefrac{3}{5} \\
\end{array}
\right)
$$
The size of its  support is $N(A_{19}) = 65$ and the permanent $\per (A_{19}) =  \frac{223}{25} =  8.92$.

\end{document}